\documentclass[11pt,a4papar]{article}

\usepackage{amsthm}
\usepackage{amssymb}
\usepackage{amscd}
\usepackage{amsmath}
\usepackage[all]{xy}
\usepackage[top=30truemm,bottom=30truemm,left=25truemm,right=25truemm]{geometry}
\usepackage{comment}
\usepackage{algorithmic,algorithm}
\usepackage{here}
\usepackage{graphicx}
\usepackage{color} %
\usepackage{enumerate} %
\usepackage{listliketab}
\usepackage{multirow}
\usepackage{cases}


\makeatletter
  
  \@addtoreset{algorithm}{subsection}
\makeatother

\theoremstyle{definition}

\theoremstyle{definition}

\theoremstyle{plain}
\newtheorem{theo}{Theorem}

\theoremstyle{plain}
\newtheorem{theor}{Theorem}

\theoremstyle{plain}
\theoremstyle{plain}

\theoremstyle{plain}

\theoremstyle{plain}
\newtheorem{thm}{Theorem}[subsection]

\theoremstyle{definition}

\theoremstyle{definition}

\theoremstyle{definition}

\theoremstyle{definition}

\theoremstyle{definition}
\newtheorem{rem}[thm]{Remark}

\theoremstyle{plain}
\newtheorem{prop}[thm]{Proposition}

\theoremstyle{plain}
\newtheorem{lem}[thm]{Lemma}

\theoremstyle{plain}
\newtheorem{cor}[thm]{Corollary}

\theoremstyle{definition}

\theoremstyle{definition}

\theoremstyle{plain}

\theoremstyle{definition}

\theoremstyle{plain}
\newtheorem{prob}[thm]{Problem}

\numberwithin{equation}{section}

\def\F{\mathbb{F}}

\begin{document}

\title{A note on certain superspecial and maximal curves of genus $5$}
\author{Momonari Kudo\thanks{Kobe City College of Technology.}
\thanks{Institute of Mathematics for Industry, Kyushu University. E-mail: \texttt{m-kudo@math.kyushu-u.ac.jp}}}
\providecommand{\keywords}[1]{\textbf{\textit{Key words---}} #1}
\maketitle

\begin{abstract}
In this note, we characterize certain maximal superspecial curves of genus $5$ over finite fields.
Specifically, we prove that the desingularization $T_p$ of $x y z^3 +  x^5 + y^5 = 0$ is a maximal superspecial trigonal curve of genus $5$ {\it if and only if} $p \equiv -1 \pmod{15}$ or $p \equiv 11 \pmod{15}$.
Moreover, we give families of maximal curves of more general type, which include $T_p$.
\end{abstract}


\section{Introduction}\label{sec:intro}

By a curve, we mean a projective, geometrically irreducible, non-singular algebraic curve.
Let $K$ be a field of positive characteristic $p > 0$.
A curve $C$ of genus $g$ over $K$ is said to be {\it superspecial} if its Jacobian variety is isomorphic to the product of supersingular elliptic curves over the algebraic closure $\overline{K}$ of $K$.
Any superspecial curve is a {\it supersingular} curve, which is a curve such that its Jacobian variety is isogenous to the product of supersingular elliptic curves.
Furthermore, superspecial curves are closely related to {\it maximal} curves and {\it minimal} curves, where a curve over $\mathbb{F}_q$ is called a maximal (resp.\ minimal) curve if the number of its $\mathbb{F}_q$-rational points attains the Hasse-Weil upper (resp.\ lower) bound $q + 1 + 2 g \sqrt{q}$ (resp.\ $q + 1 - 2 g \sqrt{q}$).
It is known that a superspecial curve over an algebraically closed field in characteristic $p$ descends to a maximal or minimal curve over $\mathbb{F}_{p^2}$.
In contrast, any maximal or minimal curve over $\mathbb{F}_{p^2}$ is superspecial.
Note that superspecial curves over $\mathbb{F}_{p^2}$ are not necessarily $\mathbb{F}_{p^2}$-maximal.
This note studies superspecial curves in order to find $\mathbb{F}_{p^2}$-maximal curves.

Finding maximal curves is an active research problem in the study of curves over finite fields, see e.g., \cite{Kawakita}, \cite{KTT}, \cite{Taf16}, \cite{Taf}, \cite{Taf3}, \cite{Taf2}, and \cite{Taf18}.
In particular, Tafazolian and Torres (cf.\ \cite{Taf16}, \cite{Taf}, \cite{Taf3}, \cite{Taf2}, \cite{Taf18}) characterize maximal curves of various types, e.g., hyperelliptic curves $y^2 = x^m + x$ and $y^2 = x^m + 1$, Fermat type $x^m + y^n = 1$, Hurwitz type $x^m y^a + y^n + x^b = 0$, other types such as $y^n = x^{\ell} (x^m + 1)$.
They use Serre's covering result (cf.\ \cite[Prop.\ 2.3]{KNT}, \cite{Lach}):
Any curve over $\mathbb{F}_{q^2}$ non-trivially $\mathbb{F}_{q^2}$-covered by a maximal curve over $\mathbb{F}_{q^2}$ is also a maximal curve over $\mathbb{F}_{q^2}$.
They prove the maximality of curves by finding subcovers of a maximal curve such as the Hermitian curve.

Another problem is the {\it enumeration} of superspecial (or maximal) curves.
Enumeration means here that we count the number of $K$-isomorphism classes of superspecial (or maximal) curves of genus $g$ over $\mathbb{F}_q$ for $K=\mathbb{F}_q$ or $\overline{\mathbb{F}_q}$.
If $g \leq 3$, some theoretical approaches to enumerate superspecial (or maximal) curves are available, and they are based on Torelli's theorem (cf.\ \cite{Deuring}, \cite[Prop.\ 4.4]{XYY16} for $g=1$, \cite{HI}, \cite{IK}, \cite{Serre1983} for $g=2$, and \cite{Hashimoto}, \cite{Ibukiyama} for $g=3$).
If $g \geq 4$, however, it is thought that these approaches are not so effective; different from the case of $g \leq 3$, the dimension of the moduli space of curves of genus $g$ is strictly less than that of the moduli space of principally polarized abelian varieties of dimension $g$.

To deal with the case of $g \geq 4$, alternative approaches based on both theory of algebraic geometry and computer algebraic methods are proposed in \cite{KH17-2}, \cite{KH17a}, \cite{KH16}, \cite{KH18a} and \cite{KH18}.
In \cite{KH18}, superspecial trigonal curves of genus $5$ over $\mathbb{F}_{q}$ are completely enumerated for $q = 5$, $5^2$, $7$, $7^2$, $11$ and $13$.
In particular, there are precisely four $\mathbb{F}_{11}$-isomorphism classes of superspecial trigonal curves.
Enumerated isomorphism classes are represented by the non-singular models of the singular plane curves $x y z^3 + s x^5 + t y^5  = 0$ with a singularity $[0:0:1]$ in $\mathbf{P}^2= \mathrm{Proj} (\overline{\mathbb{F}_{11}}[x,y,z])$ for some $s \in \mathbb{F}_{11}^{\times}$ and $t \in \mathbb{F}_{11}^{\times}$.
The non-singular model with $s=t=1$ is also an $\mathbb{F}_{11^2}$-maximal curve, see \cite[Remark 5.1.2]{KH18}.

In this note, we investigate the superspeciality of the non-singular model of $x y z^3 + s x^5 + t y^5 = 0$ with $s \in \mathbb{F}_{q}^{\times}$ and $t \in \mathbb{F}_{q}^{\times}$ for {\it arbitrary} $q$.
We also investigate the $\mathbb{F}_{p^2}$-maximality of the curve with $(s,t)=(1,1)$ for arbitrary $p > 11$.
Moreover, we consider curves of more general type: the non-singular model of the homogeneous equation $x^a + y^a + z^b x^c y^c = 0$ for natural numbers $a$, $b$ and $c$ with $b + 2 c = a$.
Main results of this note are as follows:

\begin{theo}
Let $(s, t)$ be a pair of elements in $\mathbb{F}_{q}^{\times}$.
Put $F := x y z^3 + s x^5 + t y^5$.
Let $ V(F)$ denote the projective zero-locus in $\mathbf{P}^2 = \mathrm{Proj}(\overline{\mathbb{F}_{q}}[x,y,z])$ defined by $F=0$.
Then the desingularization $T_p$ of $V(F)$ is a superspecial trigonal curve of genus $5$ if and only if $p \equiv 2 \pmod{3}$ and $p \equiv 1, 4 \pmod{5}$.
\end{theo}

We directly prove Theorem \ref{thm:main-2} by computing the Hasse-Witt matrix of $T_p$.
Since $\mathbb{F}_{p^2}$-maximal curves are superspecial, the condition $p \equiv 2 \pmod{3}$ and $p \equiv 1, 4 \pmod{5}$ is necessary for $T_p$ to be $\mathbb{F}_{p^2}$-maximal.
In fact, the converse is true for $(s,t)=(1,1)$.

\begin{theo}
Let $V(F)$ denote the projective zero-locus in $\mathbf{P}^2$ defined by $F=x y z^3 + x^5 + y^5=0$.
Then the desingularization $T_p$ of $V(F)$ is an $\mathbb{F}_{p^2}$-maximal trigonal curve of genus $5$ if and only if $p \equiv -1 \pmod{15}$ or $p \equiv 11 \pmod{15}$.
\end{theo}

We prove Theorem \ref{thm:general} below by constructing explicit equations for subcovers of the Hermitian curve, similarly to methods in \cite{Taf16}, \cite{Taf}, \cite{Taf3}, \cite{Taf2} and \cite{Taf18}.

\begin{theo}
Let $a$, $b$ and $c$ be natural numbers with $b + 2 c = a$.
Put $n := a b / \mathrm{gcd}(a,b,c)$.
Let $C_{p,a,b,c}$ be the non-singular model of $x^a + y^a + z^b x^c y^c = 0$.
Then we have the following:
\begin{enumerate}
\item[$(1)$] The curve $C_{p,a,b,c}$ is maximal over $\mathbb{F}_{p^2}$ if $p \equiv -1 \pmod{n}$.
\item[$(2)$] Suppose that $a$ and $b$ are coprime.
Since there exists an integer $d$ such that $d \equiv 2 \pmod{a}$ and $d \equiv 0 \pmod{b}$, and since any two such $d$ are congruent modulo $a b$, we fix such a $d$ in the interval $[0, a b)$. 
Then the curve $C_{p,a,b,c}$ is maximal over $\mathbb{F}_{p^2}$ if $p \equiv d-1 \pmod{n}$.
\end{enumerate}
\end{theo}

\subsection*{Acknowledgments}
I thank Shushi Harashita for his comments on the preliminary version of this note.
I also thank Bradley Brock for sending me his thesis~\cite{Brock}.
From his thesis, I learned much about superspecial curves.
\section{Superspeciality of the trigonal curve $T_p$}\label{sec:pre-2}

As in the previous section, let $K$ be a perfect field of characteristic $p>2$.
Let $K[x,y,z]$ denote the polynomial ring of the three variables $x$, $y$ and $z$.
As an example of a superspecial curve of genus $g=5$ in characteristic $11$, we have the desingularization of the projective variety $x y z^3 + s x^5 + t y^5 = 0$ with $(s, t) \in (K^{\times})^2$ in the projective plane $\mathbf{P}^2 = \mathrm{Proj}(\overline{K}[x,y,z])$, see \cite[Theorem B]{KH18}.

In this section, we shall prove that the desingularization of the variety $x y z^3 + s x^5 + t y^5 = 0$ over $K$ is (resp.\ not) a superspecial curve of genus $5$ if $p \equiv 2 \pmod{3}$ and $p \equiv 1 \pmod{5}$ or if $p \equiv 2 \pmod{3}$ and $p \equiv 4 \pmod{5}$ (resp.\ otherwise).
Throughout this section, we set $F := x y z^3 + s x^5 + t y^5$.
Let $T_p$ denote the non-singular model of the projective variety $V(F)$ in $\mathbf{P}^2$ defined by $F= 0$.

\subsection{Singularity of $V(F)$}

First, we prove that the variety $V(F)$ has a unique singular point if $p > 5$.

\begin{lem}\label{lem:sing-2}
If $p > 5$ $($resp.\ $p=5)$, then the variety $V(F)$ has a unique singular point $[0:0:1]$ $($resp.\ at least two singular points$)$.
\end{lem}

\begin{proof}
Let $J (F)$ denote the set of all the elements of the Jacobian matrix
\[
\left( \begin{array}{ccc}
	\frac{\partial F}{\partial x} & \frac{\partial F}{\partial y} & \frac{\partial F}{\partial z}
	\end{array} \right)
	=
	\left( \begin{array}{ccc}
	y z^3 + 5 s x^4 & x z^3 + 5 t y^4 &  3 x y z^2  \\
	\end{array} \right) .
\]
Namely, the set $J (F)$ consists of the following $3$ elements: $f_1 = y z^3 + 5 s x^4$, $f_2 =  x z^3 + 5 t y^4$ and $f_3 :=3 x y z^2$.
Assume $p > 5$.
It suffices to show that $x$ and $y$ belong to the radical of the ideal generated by $F$ and $J (F)$.
By straightforward computations, we have
\begin{eqnarray}
F - x f_1 - 5^{-1} y f_2 + 5^{-1} x f_1 & = & - 3 s x^5 \nonumber \\
F - y f_2 - 5^{-1} x f_1 + 5^{-1} y f_2 & = & - 3 t y^5 \nonumber
\end{eqnarray}
which belong to the ideal $\langle F, J (F) \rangle$ in $K[x,y,z]$.
Thus, $x$ and $y$ belong to its radical.

If $p = 5$, then the point $[s^{-1/5}: - t^{-1/5}:0]$ on $V(F)$ is a singular point over the algebraic closure $\overline{\mathbb{F}_{5}}$ for each $(s,t) \in K^{\oplus 2} \smallsetminus \{ (0,0) \}$.
\end{proof}

\subsection{Irreducibility of $F$}
In this subsection, we show that the variety $V (F)$ with $F = x y z^3 + s x^5 + t y^5$ is irreducible for $s$ and $t \in K^{\times}$, equivalently the quintic form $F$ is irreducible over the algebraic closure $\overline{K}$.

\begin{lem}
The quintic form $F = x y z^3 + s x^5 + t y^5$ is irreducible for $s$ and $t \in K^{\times}$ over the algebraic closure.
\end{lem}

\begin{proof}
It suffices to prove that $s^{-1} F = x^5 + s^{-1} t y^{5} + s^{-1} x y z^3$, which is monic with respect to $x$, is irreducible over $\overline{K}$.
If $s^{-1} F$ is reducible, we have either of the following two cases: (1) $s^{-1} F = g_1 g_4$ for some linear form $g_1$ and some quartic form $g_4$ in $\overline{K}[x,y,z]$, or (2) $s^{-1} F = g_2 g_3$ for some quadratic form $g_2$ and some cubic form $g_3$ in $\overline{K}[x,y,z]$.
\begin{enumerate}
\item We may assume that the coefficient of $x$ (resp.\ $x^4$) in $g_1$ (resp.\ $g_4$) is $1$.
Writing 
\begin{eqnarray}
g_1 & = & x + a_1 y + a_2 z, \quad \mbox{and}   \nonumber \\
g_4 & = & x^4 + a_3 x^3 y + a_4 x^3 z + a_5 x^2 y^2 + a_6 x^2 y z + a_7 x^2 z^2 + a_8 x y^3 + a_9 x y^2 z\nonumber \\
& & + a_{10} x y z^2 + a_{11}x z^3 + a_{12} y^4 + a_{13}y^3 z + a_{14}y^2 z^2 + a_{15} y z^3 + a_{16} z^4 \nonumber 
\end{eqnarray} 
for $a_i \in \overline{K}$ with $1 \leq i \leq 16$, we have
\begin{eqnarray}
g_1 g_4 &=& x^5 + (a_1 + a_3) x^4 y + (a_2 + a_4) x^4 z + ( a_1 a_3 + a_5) x^3 y^3 + (a _2 a_3 + a_1 a_4 a_6 ) x^3 y z +  (a_{2} a_4 + a_7 ) x^3 z^2 \nonumber \\
& &  + (a_1 a_5 + a_8) x^2 y^3 + (a_2 a_5 + a_1 a_6 + a_9 ) x^2 y^2 z + (a_2 a_6 + a_1 a_7 + a_10) x^2 y z^2 + (a_2 a_7 + a_{11}) x^2 z^3 \nonumber \\
& & + (a_1 a_8 + a_{12} ) x y^4 + (a_2 a_8 + a_1 a_9 + a_{13}) x y^3 z + (a_2 a_9 + a_1 a_{10} + a_{14}) x y^2 z^2 \nonumber \\
& &  + (a_2 a_{10} + a_1 a_{11} + a_{15}) x y z^3 + (a_{2} a_{11} + a_{16}) x z^4 + a_{1} a_{12} y^5 + (a_2 a_{12} + a_1 a_{13}) y^4 z \nonumber \\
& & + (a_2 a_{13} + a_1 a_{14}) y^3 z^2 + (a_2 a_{14} + a_1 a_{15}) y^2 z^3 + (a_2 a_{15} + a_1 a_{16}) y z^4 + a_2 a_{16}z^5. \nonumber
\end{eqnarray}
Note that $a_1 \neq 0$ and $a_{12} \neq 0$ since the coefficient $a_1 a_{12}$ of $y^5$ is $s^{-1}t$.
Since the coefficient of $z^5$ in $F$ is zero, we have $a_2 a_{16} = 0$ and thus $a_2 = 0$ or $a_{16} = 0$.
Assume $a_2 = 0$.
In this case, we have $a_1 a_i = 0$ and hence $a_i = 0$ for $13 \leq i \leq 16$ since the coefficients of $y z^4$, $y^2 z^3$, $y^3 z^2$ and $y^4 z$ are zero.
The coefficients of $y^5$ and $x y z^3 $ are $a_1 a_{12} = s^{-1} t$ and $a_2 a_{10} + a_1 a_{11} + a_{15} =a_1 a_{11} = s^{-1}$ respectively, and thus $t a_{11}=a_{12}$ by $a_1 \neq 0$.
The coefficient of $x^2 z^3$ is $a_2 a_7 + a_{11}= a_{11}$ is zero, which contradicts $t a_{11} = a_{12} \neq 0$.
Assume $a_2 \neq 0$ and $a_{16} = 0$.
Since the coefficients of $y z^4$, $y^2 z^3$, $y^3 z^2$ and $y^4 z$ are zero, one has $a_i = 0$ for $12 \leq i \leq 15$.
The condition $a_{12} = 0$ is a contradiction.
\item Writing
\begin{eqnarray}
g_2 & = & x^2 + b_1 x y + b_2 x z + b_3 y^2 + b_4 y z + b_5 z^2, \quad \mbox{and} \nonumber \\
g_3 & = & x^3 + b_6 x^2 y + b_7 x^2 z + b_8 x y^2 + b_9 x y z + b_{10}x z^2 + b_{11}y^3 + b_{12}y^2 z + b_{13}y z^2 + b_{14} z^3 \nonumber
\end{eqnarray}
for $b_i \in \overline{K}$ with $1 \leq i \leq 14$, we have
\begin{eqnarray}
g_2 g_3 & = & x^5 + (b_1 + b_6) x^4 y + (b_2 + b_7) x^4 z + (b_1 b_6 + b_3 + b_8) x^3 y^2 \nonumber \\
& & + (b_2 b_6 + b_1 b_7 + b_4 + b_9) x^3 y z + (b_2 b_7 + b_5 + b_{10}) x^3 z^2 \nonumber \\
& & + (b_3 b_6 + b_1 b_8 + b_{11}) x^2 y^3 + (b_4 b_6 + b_3 b_7 + b_2 b_8 + b_1 b_9 + b_{12})x^2 y^2 z \nonumber \\
& & + (b_5 b_6 + b_4 b_7 + b_2 b_9 + b_1 b_{10} + b_{13}) x^2 y z^2 + (b_{5} b_{7} + b_{2}b_{10} + b_{14}) x^2 z^3 \nonumber \\
& & + (b_3 b_8 + b_1 b_{11}) x y^4 + (b_4 b_8 + b_3 b_9 + b_2 b_{11} + b_1 b_{12}) x y^3 z \nonumber \\
& & + (b_5 b_8 + b_4 b_9 + b_3 b_{10} + b_2 b_{12} + b_1 b_{13}) x y^2 z^2 \nonumber \\
& & + (b_5 b_9 + b_4 b_{10} + b_2 b_{13} + b_1 b_{14})x y z^3 + (b_5 b_{10} + b_{2}b_{14}) x z^4 \nonumber \\
& &  + b_3 b_{11} y^5 + (b_4 b_{11} + b_3 b_{12}) y^4 z + (b_5 b_{11} + b_4 b_{12} + b_3 b_{13}) y^3 z^2 \nonumber \\
& &+ (b_5 b_{12} + b_4 b_{13}+ b_3 b_{14}) y^2 z^3 + (b_5 b_{13} + b_4 b_{14}) y z^4 + b_5 b_{14} z^5 . \nonumber 
\end{eqnarray}
Note that $b_3 \neq 0$ and $b_{11} \neq 0$ since the coefficient $b_3 b_{11}$ of $y^5$ is $s^{-1}t$.
Since the coefficient of $z^5$ in $F$ is zero, we have $b_5 b_{14} = 0$ and thus $b_5 = 0$ or $b_{14} = 0$.
Assume $b_5 \neq 0$ and $b_{14} = 0$.
Since the coefficients of $yz^4$, $y^2 z^3$ and $y^3 z^2$ are zero, we have $b_{13} = b_{12} = b_{11} = 0$, which contradicts $b_{11} \neq 0$.
Assume $b_5 = 0$ and $b_{14} \neq 0$.
Since the coefficients of $yz^4$, $y^2 z^3$ and $y^3 z^2$ are zero, we have $b_{4} = b_{3} = 0$, which contradicts $b_{3} \neq 0$.
Assume $b_5 = b_{14} = 0$.
Since the coefficient of $y^2 z^3$ is zero, we have $b_4 b_{13}=0$.
If $b_4 = 0$, then we have $b_{13}=0$ since the coefficient of $y^3 z^2$ is $b_5 b_{11} + b_4 b_{12} + b_3 b_{13} = b_3 b_{13}$ and since $b_{3} \neq 0$.
Now we have $b_4 = b_5 = b_{13}= b_{14}=0$, which contradicts that the coefficient of $x y z^3$ is not zero.
If $b_{13} = 0$ and $b_4 \neq 0$, then we have $b_{12} = 0$ since the coefficient of $y^3 z^2$ is $b_5 b_{11} + b_4 b_{12} + b_3 b_{13} = b_4 b_{12}$ and since $b_{4} \neq 0$.
Since the coefficient $b_4 b_{11} + b_3 b_{12} = b_4 b_{11}$ of $y^4 z$ is zero, we have $b_{11} = 0$, which is a contradiction for $b_{11} \neq 0$.
\end{enumerate}
\end{proof}

\subsection{Superspeciality of $T_p$}
In the following, we suppose $p > 5$.
It is shown in \cite{KH18} that we can decide whether the desingularization $T_p$ of $V(F)$ is superspecial or not by computing the coefficients of certain monomials in $(F)^{p-1}$, where $F = x y z^3 + s x^5 + t y^5$.

\begin{prop}[\cite{KH18}, Corollary 3.1.6]\label{cor:HW-2}
With notation as above, the desingularization $T_p$ of $V(F)$ is superspecial if and only if the coefficients of all the following $25$ monomials of degree $5 (p-1)$ in $(F)^{p-1}$ are zero:
\begin{equation}
\begin{array}{ccccc}
( x^3 y z)^{p-1}, & x^{3 p-1} y^{p-3} z^{p-1}, & x^{3 p-2} y^{p-2} z^{p - 1}, &  x^{3 p - 2} y^{p-1} z^{p - 2}, & x^{3 p - 1} y^{p-2} z^{p-2}, \\
x^{p-3} y^{3 p-1} z^{p-1}, & ( x y^3 z )^{p-1} , & x^{p-2} y^{3 p-2} z^{p - 1}, &  x^{p -2} y^{3 p-1} z^{p - 2}, & x^{p-1} y^{3 p -2} z^{p-2}, \\
x^{2 p-3} y^{2 p-1} z^{p - 1}, & x^{2 p-1} y^{2 p-3} z^{p-1}, & ( x^2 y^2 z)^{p-1} , &  x^{2 p -2} y^{2 p-1} z^{p - 2}, & x^{2 p - 1} y^{2 p - 2} z^{p-2}, \\
 x^{2 p -3} y^{p-1} z^{2 p - 1}, & x^{2 p-1} y^{p-3} z^{2 p-1}, & x^{2 p-2} y^{p-2} z^{2 p - 1}, & ( x^2 y z^2 )^{p-1} , & x^{2 p- 1} y^{p-2} z^{2 p -2}, \\
 x^{p -3} y^{2 p-1} z^{2 p - 1}, & x^{p-1} y^{2 p-3} z^{2 p-1}, & x^{p-2} y^{2 p-2} z^{2 p - 1}, & x^{p-2} y^{2 p-1} z^{2 p - 2} ,& (x y^2 z^2)^{p-1}.
\end{array} \nonumber
\end{equation}
\end{prop}

To prove Theorem \ref{thm:main-2} stated in Section \ref{sec:intro} (and in Section \ref{sec:main-2}), we compute the $25$ coefficients given in Proposition \ref{cor:HW-2}.
We have
\begin{eqnarray}
(F)^{p-1} & = & 
\sum_{a + b + c = p-1} \binom{p-1}{a,b,c} ( x y z^3 )^{a} (s x^5 )^b (t y^5)^c \nonumber \\
& = & \sum_{a + b + c = p-1} \binom{p-1}{a,b,c} ( x^{a} y^{a} z^{3 a} ) (s^{b} x^{5 b}) (t^{c} y^{5 c}) \nonumber \\
& = & \sum_{a + b + c = p-1} s^b t^c \cdot \binom{p-1}{a,b,c} x^{a + 5 b} y^{a + 5 c} z^{3 a} \label{eq:F}
\end{eqnarray}
by the multinomial theorem.
To express $(F)^{p-1}$ as a sum of the form
\[
(F)^{p-1} = \sum_{(i,j,k) \in \left( \mathbb{Z}_{\geq 0} \right)^{\oplus 3}} c_{i,j,k} x^i y^j z^k,
\]
we consider the linear system
\begin{eqnarray}
  \left\{
    \begin{array}{l}
     a + b + c = p-1, \\
     a + 5 b  = i, \\
     a + 5 c = j,  \\
     3 a = k , 
    \end{array}
  \right. \label{eq:system-2}
\end{eqnarray}
and put
\begin{eqnarray}
S (i,j,k) := \{ (a, b, c) \in [0,p-1]^{\oplus 3} : (a,b,c) \mbox{ satisfies } \eqref{eq:system-2} \} \label{eq:sol_set-2}
\end{eqnarray}
for each $(i,j,k) \in \left( \mathbb{Z}_{\geq 0} \right)^{\oplus 3}$.
Using the notation $S (i,j,k)$, we have
\begin{eqnarray}
(F)^{p-1} = \sum_{(i,j,k) \in \left( \mathbb{Z}_{\geq 0} \right)^{\oplus 3}} \left( \sum_{(a, b, c) \in S(i,j,k)} s^b t^c \cdot \binom{p-1}{a,b,c} \right) x^{i} y^{j} z^{k}. \label{eq:sum-2}
\end{eqnarray}

\subsubsection{Case of $p \equiv 2 \pmod{3}$}
We first consider the case of $p \equiv 2 \pmod{3}$.

\begin{lem}\label{lem:coeff_zero-2}
With notation as above, if $p \equiv 2 \pmod{3}$, the coefficients of the monomials $x^i y^j z^{p-1}$ and $x^i y^j z^{2 p-2}$ in $(F)^{p-1}$ are zero for all $(i,j) \in \left( \mathbb{Z}_{\geq 0} \right)^{\oplus 2}$.
\end{lem}

\begin{proof}
Recall from \eqref{eq:F} that the $z$-exponent of each monomial in $(F)^{p-1}$ is $3 a$, which is divided by $3$.
On the other hand, the $z$-exponents of the monomials $x^i y^j z^{p-1}$ and $x^i y^j z^{2 p-2}$ are $p-1$ and $2p-2$, which are congruent to $1$ and $2$ modulo $3$ respectively.
Thus their coefficients in $(F)^{p-1}$ are all zero.
\end{proof}


Let $\mathcal{M}$ be the set of the $25$ monomials given in Proposition \ref{cor:HW-2}, and set
\[
E (\mathcal{M}) := \{ (i,j,k ) \in \left( \mathbb{Z}_{\geq 0} \right)^{\oplus 3} : x^i y^j z^k = m \mbox{ for some } m \in \mathcal{M} \} ,
\]
which is the set of the exponent vectors of the monomials in $\mathcal{M}$.

\begin{lem}\label{lem:coeff-1-2}
Assume $p \equiv 2 \pmod{3}$.
If $p \equiv 1 \pmod{5}$ or $p \equiv 4 \pmod{5}$, then we have $S (i,j,k) = \emptyset$ for any $(i,j,k) \in E(\mathcal{M})$.
\end{lem}

\begin{proof}
Note that for each $(i,j,k) \in E(\mathcal{M})$, we have $i+j+k=5 (p-1)$, see Proposition \ref{cor:HW-2}.
Using matrices, we write the system \eqref{eq:system-2} as
\begin{eqnarray}
	\left( \begin{array}{ccc}
	1 & 1 & 1 \\
	1 & 5 & 0 \\
	1 & 0 & 5 \\
	3 & 0 & 0 
	\end{array} \right)
	\left( \begin{array}{c}
	a \\
	b \\
	c 
	\end{array} \right)
	= \left( \begin{array}{c}
	p-1 \\
	i \\
	j \\
	k
	\end{array} \right) , \label{eq:sys1-1-2}
\end{eqnarray}
whose extended coefficient matrix is transformed as follows:
\[
\left( \begin{array}{cccc}
	1 & 1 & 1 & p-1\\
	1 & 5 & 0 & i\\
	1 & 0 & 5 & j \\
	3 & 0 & 0 & k 
	\end{array} \right)
\longrightarrow
	\left( \begin{array}{cccc}
	1 & 1 & 1 & p - 1\\
	0 & 1 & -4 & - j + ( p - 1 ) \\
	0 & 0 & 15 & i + 4 j -5 ( p - 1 ) \\
	0 & 0 & 0 & 0 
	\end{array} \right)
\]
Considering modulo $5$, we have the following linear system over $\mathbb{F}_5$:
\begin{eqnarray}
\left( \begin{array}{ccc}
	1 & 1 & 1\\
	0 & 1 & 1\\
	0 & 0 & 0\\
	0 & 0 & 0 
	\end{array} \right)
	\left( \begin{array}{c}
	a^{\prime} \\
	b^{\prime} \\
	c^{\prime} 
	\end{array} \right)
	= \left( \begin{array}{c}
	p-1\\
	-j + (p-1)\\
	i - j \\
	0 
	\end{array} \right) . \label{eq:sys-2}
\end{eqnarray}
Note that the system \eqref{eq:sys-2} over $\mathbb{F}_5$ has a solution if and only if $i - j \equiv 0 \pmod{5}$.
Assume $p \equiv 2 \pmod{3}$.
We claim that if $p \equiv 1 \pmod{5}$ or if $p \equiv 4 \pmod{5}$, the original system \eqref{eq:sys1-1-2} over $\mathbb{Z}$ has no solution in $[0,p-1]^{\oplus 3}$ for any $(i,j,k) \in E (\mathcal{M})$.
Indeed, if $p \equiv 1 \pmod{5}$ or $p \equiv 4 \pmod{5}$, and if the system \eqref{eq:sys1-1-2} has a solution in $[0,p-1]^{\oplus 3}$ for some $(i,j,k) \in E (\mathcal{M})$, the system \eqref{eq:sys-2} has a solution.
Since $p \equiv 2 \pmod{3}$, it follows from Lemma \ref{lem:coeff_zero-2} that $k \neq p-1$ and $k \neq 2 p-2$, i.e., $k = 2 p -1$ or $k= p-2$.
Thus we may assume that $(i,j,k)$ is either of the following:
\[
( 3 p - 2, p - 1, p - 2 ), \quad ( 3 p - 1, p - 2, p - 2 ), \quad ( p - 2, 3 p - 1, p - 2 ), \quad ( p - 1, 3 p - 2, p - 2),
\]
\[
(2 p - 2, 2 p - 1, p - 2 ), \quad ( 2 p - 1, 2 p - 2, p - 2 ), \quad ( 2 p - 3, p - 1, 2 p - 1 ), \quad ( 2 p - 1, p - 3, 2 p - 1 ), 
\]
\[
(2 p - 2, p - 2, 2 p - 1), \quad ( p - 3, 2 p - 1, 2 p - 1), \quad ( p - 1, 2 p - 3, 2 p - 1), \quad ( p - 2, 2 p - 2, 2 p - 1 ) .
\]
The value $i - j$ takes $2 p - 1$, $2p + 1$, $- 2 p - 1$, $- 2 p + 1$, $-1$, $1$, $p-2$, $p+2$, $p$, $-p-2$ or $-p+2$, $- p$, each of which is not congruent to $0$ modulo $5$ if $p \equiv 1 \pmod{5}$ or if $p \equiv 4 \pmod{5}$.
This is a contradiction.
\end{proof}

\begin{prop}\label{prop:main1-2}
Assume $p \equiv 2 \pmod{3}$.
If $p \equiv 1 \pmod{5}$ or $p \equiv 4 \pmod{5}$, then the desingularization $T_p$ of $V(F)$ is superspecial.
\end{prop}

\begin{proof}
It follows from Lemma \ref{lem:coeff-1-2} that the coefficient of $x^{i} y^{j} z^{k}$ in \eqref{eq:sum-2} is zero for each $(i,j,k) \in E(\mathcal{M})$. 
By Proposition \ref{cor:HW-2}, the desingularization of $V(F)$ is superspecial.
\end{proof}

It follows from the proof of Lemma \ref{lem:coeff-1-2} that \eqref{eq:system-2} is equivalent to the following system:
\begin{eqnarray}
  \left\{
    \begin{array}{l}
     a + b + c = p-1, \\
     b - 4 c  = - j + ( p - 1 ),  \\
     15 c = i + 4 j -5 ( p - 1 ).
    \end{array}
  \right. \label{eq:system3-2}
\end{eqnarray}

Different from the case where $p \equiv 1 \pmod{5}$ or $p \equiv 4 \pmod{5}$, the desingularization of $V(F)$ is not superspecial if $p \equiv 3 \pmod{5}$.

\begin{lem}\label{lem:coeff-2-2}
Assume $p \equiv 2 \pmod{3}$.
If $p \equiv 3 \pmod{5}$, then we have $\# S (3 p- 2, p - 1, p-2) = 1$.
In other words, the system \eqref{eq:system3-2} with $(i,j,k) = (3 p- 2, p - 1, p-2) $ has a unique solution in $[0,p-1]^{\oplus 3}$.
The solution is given by
\begin{eqnarray}
\left( \begin{array}{ccc}
a, & b, & c
\end{array} \right)
 = 
\left( \begin{array}{ccc}
 (p-2)/3, & 4 (2 p -1)/ 15, & (2 p -1)/ 15
 \end{array} \right) . \label{eq:sol-2-2}
\end{eqnarray}
Note that $(p-2)/3$, $4 (2 p - 1) / 15$ and $(2 p - 1)/15$ are less than $p-1$.
\end{lem}

\begin{proof}
The system to be solved with $(i,j,k) = (3 p- 2, p - 1, p-2)$ is given by
\begin{numcases}
  {}
  a + b + c = p-1, & \label{eq:1-2-2}\\
  b - 4 c = 0, & \label{eq:2-2-2}\\
  15 c =  2 p - 1 & \label{eq:3-2-2} 
\end{numcases}
with $(a, b, c) \in [0,p-1]^{\oplus 3}$.
Note that $2 p - 1 \equiv 2 \cdot 2 - 1  \equiv 0 \pmod{3}$ and $2 p - 1 \equiv 2 \cdot 3 - 1  \equiv 0 \pmod{5}$, and thus $2 p - 1$ is divided by $15$.
We have $c = (2 p -1 )/ 15$ by \eqref{eq:3-2-2} and $b = 4 c = 4 (2 p -1)/ 15$ by \eqref{eq:2-2-2}.
Since $b + c = (2 p - 1)/3$, it follows from \eqref{eq:1-2-2} that $a = (p-2)/3$.
\end{proof}

\begin{lem}\label{lem:coeff2-2-2}
Assume $p \equiv 2 \pmod{3}$.
If $p \equiv 3 \pmod{3}$, then the coefficient of the monomial $x^{3 p-2} y^{p-1} z^{p-2}$ in $(F)^{p-1}$ is not zero.
\end{lem}

\begin{proof}
Let $c_{3 p- 2,p-1,p-2}$ be the coefficient of $x^{3 p-2} y^{p-1} z^{p-2}$ in $(F)^{p-1}$.
Recall from \eqref{eq:sum-2} that $c_{3 p- 2,p-1,p-2}$ is given by
\[
\sum_{(a,b,c) \in S(3 p -2, p-1, p-2)} s^b t^c \cdot \binom{p-1}{a,b,c} ,
\]
where $S (3 p- 2,p-1,p-2) $ is defined in \eqref{eq:sol_set-2}.
By Lemma \ref{lem:coeff-2-2}, the set $S(3 p- 2,p-1,p-2)$ consists of only the element given by \eqref{eq:sol-2-2}, and hence
\[
c_{3 p- 2,p-1,p-2} = \cfrac{(p-1)!}{\left( \cfrac{p-2}{3} \right) ! \left( \cfrac{4(2 p-1)}{15} \right) ! \left( \cfrac{2 p-1}{15} \right) !},
\]
which is not divisible by $p$.
\end{proof}


\begin{prop}\label{prop:main-2-2}
Assume $p \equiv 2 \pmod{3}$.
If $p \equiv 3 \pmod{5}$, then the desingularization $T_p$ of $V(F)$ is not superspecial.
\end{prop}

\begin{proof}
It follows from Lemma \ref{lem:coeff2-2-2} that the coefficient of $x^{3p-2} y^{p-1} z^{p-2}$ in $(F)^{p-1}$ is not zero.
By Proposition \ref{cor:HW-2}, the desingularization of $V(F)$ is not superspecial.
\end{proof}

\subsubsection{Case of $p \equiv 1 \pmod{3}$}
Next, we consider the case of $p \equiv 1 \pmod{3}$.

\begin{lem}\label{lem:coeff-2}
Assume $p \equiv 1 \pmod{3}$.
Then we have $\# S (2 p- 2, 2 p - 2, p-1) = 1$.
In other words, the system \eqref{eq:system3-2} with $(i,j,k) = (2 p- 2, 2 p - 2, p-1)$ has a unique solution in $[0,p-1]^{\oplus 3}$.
The solution is given by
\begin{eqnarray}
\left( \begin{array}{ccc}
a, & b, & c
\end{array} \right)
 = 
\left( \begin{array}{ccc}
 (p-1)/3, & (p-1)/3, & (p-1)/3
 \end{array} \right) . \label{eq:sol-2}
\end{eqnarray}
\end{lem}

\begin{proof}
The system to be solved with $(i,j,k) = (2 p- 2, 2 p - 2, p-1)$ is given by
\begin{numcases}
  {}
  a + b + c = p-1, & \label{eq:1-2}\\
  b - 4 c = - ( p - 1), & \label{eq:2-2}\\
  15 c =  5 ( p - 1 ) & \label{eq:3-2} 
\end{numcases}
with $(a, b, c) \in [0,p-1]^{\oplus 3}$.
Since $p - 1$ is divided by $3$ from our assumption, it follows from \eqref{eq:3-2} that $c = (p-1)/3$.
By \eqref{eq:2-2} and \eqref{eq:1-2}, we have $a = b = (p-1)/3$.
\end{proof}

\begin{lem}\label{lem:coeff2-2}
Assume $p \equiv 1 \pmod{3}$.
Then the coefficient of the monomial $x^{2 p-2} y^{2p-2} z^{p-1}$ in $(F)^{p-1}$ is not zero.
\end{lem}

\begin{proof}
Let $c_{2 p- 2,2p-2,p-1}$ be the coefficient of $x^{2 p- 2} y^{2 p-2 } z^{p-1}$ in $(F)^{p-1}$.
Recall from \eqref{eq:sum-2} that $c_{2p-2, 2p-2, p-1}$ is given by
\[
\sum_{(a,b,c) \in S(2 p -2, 2 p-2, p-1)} s^b t^c \cdot \binom{p-1}{a,b,c} ,
\]
where $S (2 p-2 ,2p-2, p-1)$ is defined in \eqref{eq:sol_set-2}.
By Lemma \ref{lem:coeff-2}, the set $S(2 p -2, 2 p - 2, p-1)$ consists of only the element given by \eqref{eq:sol-2}, and hence
\[
c_{2 p-2, 2 p-2, p -1} = \cfrac{(p-1)!}{\left( \cfrac{p-1}{3} \right) ! \left( \cfrac{p-1}{3} \right) ! \left( \cfrac{p-1}{3} \right) !},
\]
which is not divisible by $p$.
\end{proof}


\begin{prop}\label{prop:main-2}
Assume $p \equiv 1 \pmod{3}$.
Then the desingularization $T_p$ of $V(F)$ is not superspecial.
\end{prop}

\begin{proof}
It follows from Lemma \ref{lem:coeff2-2} that the coefficient of $x^{2p-2} y^{2p-2} z^{p-1}$ in $(F)^{p-1}$ is not zero.
By Proposition \ref{cor:HW-2}, the desingularization $T_p$ of $V(F)$ is not superspecial.
\end{proof}

\subsubsection{Proofs of the superspeciality of $T_p$}\label{sec:main-2}

\begin{theor}\label{thm:main-2}
Let $(s, t)$ be a pair of elements in $\mathbb{F}_{q}^{\times}$.
Put $F := x y z^3 + s x^5 + t y^5$.
Let $ V(F)$ denote the projective zero-locus in $\mathbf{P}^2 = \mathrm{Proj}(\overline{K}[x,y,z])$ defined by $F=0$.
Then the desingularization $T_p$ of $V(F)$ is a superspecial trigonal curve of genus $5$ if and only if $p \equiv 2 \pmod{3}$ and $p \equiv 1, 4 \pmod{5}$.
\end{theor}

\begin{proof}
Since $F$ is an irreducible quintic form over $K$, the desingularization $T_p$ of $V(F)$ is a trigonal curve of genus $5$ over $K$, see \cite[Section 2]{KH18}.
The assertion follows from Propositions \ref{prop:main1-2}, \ref{prop:main-2-2} and \ref{prop:main-2}.
\end{proof}

\begin{cor}\label{cor:main-2}
There exist superspecial trigonal curves of genus $5$ in characteristic $p$ for infinitely many primes $p$.
The set of primes $p$ for which $T_p$ is superspecial has natural density $1/4$. 
\end{cor}

\begin{proof}
Note that $p \equiv 2 \pmod{3}$ and $p \equiv 1 \pmod{5}$ (resp.\ $p \equiv 2 \pmod{3}$ and $p \equiv 4 \pmod{5}$) is equivalent to $p \equiv 11 \pmod{15}$ (resp.\ $p \equiv 14 \pmod{15}$).
Since both $11$ and $14$ are coprime to $15$, it follows from Dirichlet's Theorem that there are infinitely many primes congruent to $11$ or $14$ modulo $15$.
Thus, the first claim follows from Theorem \ref{thm:main-2}.
The second claim is deduced from the fact that the natural density of primes equal to $11$ or $14$ modulo $15$ is 
\[
\cfrac{1}{\varphi (15)} + \cfrac{1}{\varphi (15)} = \cfrac{1}{8} + \cfrac{1}{8} = \cfrac{1}{4},
\]
where $\varphi$ is Euler's totient function.
\end{proof}

\begin{prob}
Does there exist a superspecial trigonal curve of genus $5$ in characteristic $p$ for each of the following case?
\begin{enumerate}
\item[$(1)$] $p \equiv 1 \pmod{3}$.
Cf.\ the non-existence for $p=7$ $($and $13$ over $\mathbb{F}_{13})$ is already shown in {\rm \cite{KH18}}.
\item[$(2)$] $p \equiv 2 \pmod{3}$ and $p \equiv 3 \pmod{5}$.
\end{enumerate}
\end{prob}

\subsection{Application: Finding maximal curves over $K = \mathbb{F}_{p^2}$ for large $p$}\label{sec:comp-2}
In the following, we set $K := \mathbb{F}_{p^2}$.
It is known that any maximal or minimal curve over $\mathbb{F}_{p^2}$ is supersepcial.
Conversely, any superspecial curve over an algebraically closed field descends to a maximal or minimal curve over $\mathbb{F}_{p^2}$, see the proof of \cite[Proposition 2.2.1]{KH16}.
Under the condition given in Theorem \ref{thm:main-2}, we computed the number of $\mathbb{F}_{p^2}$-rational points on $T_p$ with $s=t=1$ for $7 \leq p \leq 1000$ using a computer algebra system Magma \cite{Magma}.
Table \ref{table:3-2} shows our computational results for $7 \leq p \leq 179$.
We see from Table \ref{table:3-2} that any superspecial $T_p$ is maximal over $\mathbb{F}_{p^2}$ for $7 \leq p \leq 179$ (also for $180 \leq p \leq 1000$, but omit to write them in the table).
From our computational results, let us give a conjecture on the existence of $\mathbb{F}_{p^2}$-maximal (trigonal) curves of genus $5$.
\begin{itemize}
\item For any $p$ with $p \equiv 2 \pmod{3}$ and $p \equiv 1, 4 \pmod{5}$, the desingularization $T_p$ of $V(F)$ over $\mathbb{F}_{p^2}$ is maximal.
\end{itemize}
In the next subsection we prove this conjecture (in fact, the condition $p \equiv 2 \pmod{3}$ and $p \equiv 1, 4 \pmod{5}$ is a necessary and sufficient condition for the maximality of $T_p$).

\renewcommand{\arraystretch}{1.5}
\begin{table}[t]
\centering{
\caption{The number of $\mathbb{F}_{p^2}$-rational points on the desingularization $T_p$ of $ V(F)$ for $7 \leq p \leq 100$ with $p \equiv 2 \pmod{3}$, where $F = x y z^3 + x^5 + y^5$.
We denote by $\#T_p (\mathbb{F}_{p^2})$ the number of $\mathbb{F}_{p^2}$-rational points on $T_p$ for each $p$.
}
\label{table:3-2}
\scalebox{0.97}{
\begin{tabular}{c||c|c|c||c||c|c|c} \hline
$p$ & $p \bmod{5}$ & S.sp.\ or not & $\# T_p ( \mathbb{F}_{p^2} )$  & $p$ & $p \bmod{5}$ & S.sp.\ or not & $\# T_p ( \mathbb{F}_{p^2} )$ \\ \hline
$11$ & $1$ & S.sp. & $232$ (Max.)             &  $89$ & $4$ & S.sp. & $8812$ (Max.) \\ \hline
$29$ & $4$ & S.sp. & $1132$ (Max.)    &  $101$ & $1$ & S.sp.  & $11212$ (Max.)\\ \hline
$41$ & $1$ &  S.sp. & $2092$ (Max.) & $131$ & $1$ & S.sp. &  $18472$ (Max.) \\ \hline
 $59$ & $4$ & S.sp. & $4072$ (Max.) & $149$  &  $4$ & S.sp. & $23692$ (Max.) \\ \hline
$71$ & $1$ & S.sp. & $5752$ (Max.)   &       $179$    & $4$ & S.sp.  & $33832$ (Max.) \\ \hline
\end{tabular}
}}
\end{table}



\section{Main results}

Let $T_p$ denote the non-singular model of $x^5 + y^5 + x y z^3 = 0$.
In this section, we prove that $T_p$ is $\mathbb{F}_{p^2}$-maximal if and only if $p \equiv 2 \pmod{3}$ and $p \equiv 1, 4 \pmod{5}$.
Furthermore, we give a family of $\mathbb{F}_{p^2}$-maximal curves defined by equations of a more general form.

\subsection{Maximality of $T_p : x^5 + y^5 + x y z^3 = 0$}

\begin{theor}\label{thm:genus5}
Let $V(F)$ denote the projective zero-locus in $\mathbf{P}^2$ defined by $F=x y z^3 + x^5 + y^5=0$.
Then the desingularization $T_p$ of $V(F)$ is an $\mathbb{F}_{p^2}$-maximal trigonal curve of genus $5$ if and only if $p \equiv -1 \pmod{15}$ or $p \equiv 11 \pmod{15}$.
\end{theor}

\begin{proof}
First, we suppose that $T_p$ is $\mathbb{F}_{p^2}$-maximal.
Since an $\mathbb{F}_{p^2}$-maximal curve is superspecial, the curve $T_p$ is superspecial.
By Theorem \ref{thm:main-2}, we have $p \equiv -1 \pmod{15}$ or $p \equiv 11 \pmod{15}$.

Conversely, assume $p \equiv -1 \pmod{15}$ or $p \equiv 11 \pmod{15}$.
It suffices to show that $T_p$ is covered by an $\mathbb{F}_{p^2}$-maximal curve.
Note that we have $1 + y^5 + y z^3 = 0$ for $x=1$.
\begin{description}
\item[Case of $p \equiv -1 \pmod{15}$.] There exists an integer $m$ such that $p+1 = 15m$.
It follows from the following morphism
\begin{equation}
\left\{
\begin{array}{ccc}
\mathcal{H}_{p+1}: 1 + Y^{p+1} + Z^{p+1} = 0 & \rightarrow & 1 + y^5 + y z^3 = 0  \\
(Y,Z) & \mapsto & \left( Y^{3 m}, Y^{-m}Z^{5 m} \right) 
\end{array}
\right. \nonumber
\end{equation}
that $T_p$ is covered by the maximal Hermitian curve $\mathcal{H}_{p+1}$.
Thus $T_p$ is also maximal over $\mathbb{F}_{p^2}$.
\item[Case of $p \equiv 11 \pmod{15}$.] In this case, there exists an integer $m$ such that $p+1 = 12 + 15m$.
It follows from the following morphism
\begin{equation}
\left\{
\begin{array}{ccc}
\mathcal{H}_{p+1} : Y^{p} + Y = - Z^{p+1}& \rightarrow & 1+y^5 =- y z^3    \\
(Y,Z) & \mapsto & \left( Y^{3 m+2}, Y^{- (m+1)}Z^{5 m+4} \right) 
\end{array}
\right. \nonumber
\end{equation}
that $T_p$ is covered by $\mathcal{H}_{p+1}$.
Thus $T_p$ is also maximal over $\mathbb{F}_{p^2}$.
\end{description}
\end{proof}


\subsection{A generalization of the maximal curve of the form $x^5 + y^5 + x y z^3 = 0$}

More generally, we consider the homogeneous equation $x^a + y^a + z^b x^c y^c = 0$ for natural numbers $a$, $b$ and $c$ with $b + 2 c = a$.
Let $C_{p,a,b,c}$ be the non-singular model of $x^a + y^a + z^b x^c y^c = 0$.
Note that $C_{p,a,b,c}$ has genus
\[
\frac{1}{2} \left( a b - a + 2 - \mathrm{gcd}(b,a-c) - \mathrm{gcd}(b,c) \right).
\]
Here, we give sufficient conditions under which $C_{p,a,b,c}$ is $\mathbb{F}_{p^2}$-maximal.

\begin{theor}\label{thm:general}
Let $a$, $b$ and $c$ be natural numbers with $b + 2 c = a$.
Put $n := a b / \mathrm{gcd}(a,b,c)$.
Let $C_{p,a,b,c}$ be the non-singular model of $x^a + y^a + z^b x^c y^c = 0$.
Then we have the following:
\begin{enumerate}
\item[$(1)$] The curve $C_{p,a,b,c}$ is maximal over $\mathbb{F}_{p^2}$ if $p \equiv -1 \pmod{n}$.
\item[$(2)$] Suppose that $a$ and $b$ are coprime.
Since there exists an integer $d$ such that $d \equiv 2 \pmod{a}$ and $d \equiv 0 \pmod{b}$, and since any two such $d$ are congruent modulo $a b$, we fix such a $d$ in the interval $[0, a b)$. 
Then the curve $C_{p,a,b,c}$ is maximal over $\mathbb{F}_{p^2}$ if $p \equiv d-1 \pmod{n}$.
\end{enumerate}
\end{theor}

\begin{proof}
It suffices to show that $T_p$ is covered by an $\mathbb{F}_{p^2}$-maximal curve.
For $x=1$, we have $1 + y^a + z^b y^c = 0$.
\begin{enumerate}
\item Assume $p \equiv n-1 \pmod{n}$.
There exists an integer $m$ such that $p+1 = n m$.
It follows from the following morphism
\begin{equation}
\left\{
\begin{array}{ccc}
\mathcal{H}_{p+1} : 1 + Y^{p+1} + Z^{p+1} = 0 & \rightarrow & 1 + y^a + z^b y^c = 0 \\
(Y,Z) & \mapsto & \left( Y^{\frac{b m }{\mathrm{gcd}(a,b,c)}}, Y^{\frac{- c m}{\mathrm{gcd}(a,b,c)}} Z^{\frac{a m}{\mathrm{gcd}(a,b,c)}} \right) 
\end{array}
\right. \nonumber
\end{equation}
that $C_{p,a,b,c}$ is covered by the maximal Hermitian curve $\mathcal{H}_{p+1}$, and so is $\mathbb{F}_{p^2}$-maximal.

\item Assume that $a$ and $b$ are coprime.
Let $d$ be a unique integer with $0 \leq d < a b$ such that $d \equiv 2 \pmod{a}$ and $d \equiv 0 \pmod{b}$.
Note that $c d $ and $a - 2 c$ (resp.\ $c (d-2)$ and $a$) are divisible by $b$ (resp.\ $a$).
Since $\mathrm{gcd}(a,b)=1$, the sum $cd + ( a- 2c) = c (d-2) + a$ is divisible by $ab$.
We also assume that $p \equiv d-1 \pmod{n}$, and then there exists an integer $m$ such that $p+1 = n m + d$.
It follows from the following morphism
\begin{equation}
\left\{
\begin{array}{ccc}
\mathcal{H}_{p+1} : Y^p + Y = -Z^{p+1}  & \rightarrow & 1 + y^a + z^b y^c = 0 \nonumber \\
(Y,Z) & \mapsto & \left( Y^{\frac{b m }{\mathrm{gcd}(a,b,c)} + \frac{d-2}{a}}, Y^{\frac{- c m}{\mathrm{gcd}(a,b,c)}-\frac{c(d-2)+a}{ab}} Z^{\frac{a m}{\mathrm{gcd}(a,b,c)}+\frac{d}{b}}\right) \nonumber
\end{array}
\right. \nonumber
\end{equation}
that $C_{p,a,b,c}$ is covered by the maximal Hermitian curve $\mathcal{H}_{p+1}$, and so is maximal over $\mathbb{F}_{p^2}$.
\end{enumerate}
\end{proof}

Now we can see that Theorem \ref{thm:genus5} is a special case of Theorem \ref{thm:general} for $(a, b, c) = (5,3,1)$ with $d= 15$.

\begin{rem}
We replace $y$ and $z$ in $1 + y^a + z^b y^c = 0$ with $x$ and $y$ respectively, say $1 + x^a + y^b x^c = 0$.
It follows from $\mathrm{gcd}(a,b)=1$ with $b + 2 c = a$ that $\mathrm{gcd}(b,c)=1$, and that $a$ and $b$ are odd numbers.
There exist $\ell$ and $k$ such that $k b + \ell c = 1$.
Without loss of generality, we may assume $\ell < 0$ with $0 \leq - \ell < b$.
One can check that $- \ell a + 2 \equiv 2 \pmod{a}$ and $- \ell a + 2 \equiv 0 \pmod{b}$ with $0 \leq - \ell a +2 < a b$,
and thus can take $d$ to be $- \ell a + 2$.
Considering $(x, y) \mapsto (x^{-\ell}, -y / x^k)$,
we transform $1 + x^a + y^b x^c = 0$ into $y^b = x( 1 + x^{- \ell a} )$.
Assume $p \equiv d-1 \pmod{a b}$, i.e., $p \equiv - \ell a + 1 \pmod{a b}$.
If $p \equiv - \ell a + 1 \pmod{- \ell a b}$, it follows from \cite[Proposition 4.12]{Taf2} that the curve $\mathcal{C}( b, -\ell a + 1) : y^b = x^{-\ell a + 1} + x$ is maximal over $\mathbb{F}_{p^2}$, and hence $1 + x^a + y^b x^c = 0$ is also maximal over $\mathbb{F}_{p^2}$.
\end{rem}

\subsection{Examples}

\begin{table}[t]
\centering{
\caption{Parameters for which $C_{p,a,b,c}$ are maximal and superspecial curves over $\mathbb{F}_{p^2}$.
There are two cases listed in Theorem \ref{thm:general}.
}
\label{table:3-3}
\begin{tabular}{c||c|c|c} \hline
genus & congruence & $(a,b,c)$ & case  \\ \hline
 \multirow{2}{*}{$4$} & $p \equiv 15 \pmod{16}$  & $(8,2,3)$& case (1) with $n=16$\\ \cline{2-4}
 & $p \equiv 9 \pmod{10}$  & $(10,2,4)$& case (1) with $n=10$\\ \hline
\multirow{2}{*}{$5$} & $p \equiv 14 \pmod{15}$  & \multirow{2}{*}{$(5,3,1)$} & case (1) with $n=15$\\ 
 & $p \equiv 11 \pmod{15}$  &  & case (2) with $(n,d) = (15,12)$\\ \hline
 \multirow{2}{*}{$6$} & $p \equiv 13 \pmod{14}$  & $(14,2,6)$& case (1) with $n=14$\\ \cline{2-4}
  & $p \equiv 23 \pmod{24}$  & $(12,2,5)$& case (1) with $n=24$\\ \hline
 \multirow{3}{*}{$7$} & $p \equiv 20 \pmod{21}$  & \multirow{2}{*}{$(7,3,2)$} & case (1) with $n=21$\\ 
 & $p \equiv 8 \pmod{21}$  &  & case (2) with $(n,d) = (21,9)$\\ \cline{2-4}
 & $p \equiv 8 \pmod{9}$  & $(9,3,3)$& case (1) with $n=9$\\ \hline
 \multirow{2}{*}{$8$} & $p \equiv 31 \pmod{32}$  & $(16,2,7)$& case (1) with $n=32$\\ \cline{2-4}
 & $p \equiv 17 \pmod{18}$  & $(18,2,8)$& case (1) with $n=18$\\ \hline
  $9$ & $p \equiv 23 \pmod{24}$  & $(6,4,1)$& case (1) with $n=24$\\ \hline
 \multirow{2}{*}{$10$}  & $p \equiv 39 \pmod{40}$  & $(20,2,9)$& case (1) with $n=40$\\ \cline{2-4}
  & $p \equiv 21 \pmod{22}$  & $(22,2,10)$& case (1) with $n=22$\\ \hline
\multirow{3}{*}{$11$} & $p \equiv 32 \pmod{33}$  & \multirow{2}{*}{$(11,3,4)$} & case (1) with $n=33$\\ 
 & $p \equiv 23 \pmod{33}$  &  & case (2) with $(n,d) = (33,24)$\\ \cline{2-4}
 & $p \equiv 15 \pmod{16}$  & $(8,4,2)$& case (1) with $n=16$\\ \hline
\end{tabular}
}
\end{table}

There exists an $\mathbb{F}_{p^2}$-maximal and superspecial curve for each case in Table \ref{table:3-3}.
Three concrete examples introduced below are constructed from parameters in Table \ref{table:3-3}.
We also refer to \cite{LMPT18} for the existence of (smooth) supersingular curves over $\overline{\mathbb{F}_{p}}$, and \cite{Kudo18} for the existence of (non-hyperelliptic) superspecial curves of genus $4$.
\begin{enumerate}
\item The curve $C_{p,8,2,3}$ with $(a,b,c)=(8,2,3)$, given by $x^8 + y^8 + z^2 x^3 y^3 = 0$, is an $\mathbb{F}_{p^2}$-maximal superspecial curve of genus $4$ if $p \equiv 15 \pmod{16}$, e.g., $p= 31$, $47$, $79$, $127$, $191$.
Since $31 \equiv 1 \pmod{3} $ and $31 \equiv 1 \pmod{5}$, the superspecial (and hence supersingular) curve $C_{31,8,2,3}$ is not obtained in \cite[Theorems 1.1 and 5.5]{LMPT18} nor \cite[Theorem 3.1]{Kudo18}.
The existence of a superspecial curve of genus $4$ for $p \equiv 1 \pmod{3}$ is a positive answer to \cite[Problem 3.3]{Kudo18}.
\item For $(a,b,c)=(12,2,5)$, the curve $C_{p,12,2,5}$ with $x^{12} + y^{12} + z^{2} x^{5} y^{5} = 0$ is an $\mathbb{F}_{p^2}$-maximal superspecial curve of genus $6$ if $p \equiv 23 \pmod{24}$, e.g., $p= 23$, $47$, $71$, $167$, $191$.
Since $191 \equiv 9 \pmod{13} $ and $191 \equiv 2 \pmod{7}$, the superspecial (and thus supersingular) curve $C_{191,12,2,5}$ is not obtained in \cite[Theorems 1.1 and 5.5]{LMPT18}.
\item For $(a,b,c)=(7,3,2)$, the curve $C_{p,7,3,2}$ of genus $7$ is $\mathbb{F}_{p^2}$-maximal and superspecial if $p \equiv 20 \pmod{21}$, e.g., $p=41$, $83$, $167$ or if $p \equiv 8 \pmod{21}$, e.g., $p=29$, $71$, $113$, $197$.
Note that except for $p=29$, each $p$ listed above does not satisfy any of congruences ($p \equiv 14 \bmod{15}$ and $p \equiv 15 \bmod{16}$) for $g=7$ listed in \cite[Theorems 1.1 and 5.5]{LMPT18}.
\end{enumerate}


\end{document}